\documentclass[12pt]{amsart}

\usepackage{amssymb}
\usepackage{bm}
\usepackage{amsthm}
\usepackage{graphicx}
\usepackage{psfrag}
\usepackage[usenames,dvipsnames]{xcolor}

\usepackage{subcaption}
\usepackage{enumerate}
\usepackage{url}

\usepackage{algpseudocode}

\usepackage{rotating}

\usepackage{mathdots}
\usepackage{mathtools}
\usepackage{multirow}
\usepackage{blkarray}
\usepackage{relsize}
\usepackage{hyperref}

\usepackage{tikz}
\usetikzlibrary{positioning}
\usetikzlibrary{graphs,graphs.standard}
\usepackage{makecell}

\newcommand{\excise}[1]{}

\newtheorem{thm}{Theorem}[section]
\newtheorem{lemma}[thm]{Lemma}

\newtheorem{cor}[thm]{Corollary}
\newtheorem{prop}[thm]{Proposition}

\newtheorem{mainquestion}[thm]{Main Question}

\theoremstyle{definition}
\newtheorem{alg}[thm]{Algorithm}
\newtheorem{example}[thm]{Example}

\newtheorem{remark}[thm]{Remark}
\newtheorem{defn}[thm]{Definition}

\numberwithin{equation}{section}

\newcommand{\ring}[1]{\ensuremath{\mathbb{#1}}}
\renewcommand\>{\rangle}

\newcommand\NN{\ring{N}}

\newcommand\QQ{\ring{Q}}
\newcommand\RR{\ring{R}}

\newcommand\ZZ{\ring{Z}}

\newcommand\til{\mathord\sim}

\def\ol#1{{\overline {#1}}}

\DeclareMathOperator\spann{span} 

\DeclareMathOperator\Ap{Ap} 
\DeclareMathOperator\supp{supp} 

\DeclareMathOperator\rk{rk}
\DeclareMathOperator\Sat{Sat}
\DeclareMathOperator\Row{Row}

\newcommand{\dfloor}{\lfloor \tfrac{1}{2}(m-1)\rfloor}

\newcommand{\Cm}{\mathcal C_m}

\begin{document}

\mbox{}
\title[On faces of the Kunz cone and the numerical semigroups within them]{On faces of the Kunz  cone and the \\ numerical semigroups within them}

\author[Borevitz]{Levi Borevitz}
\address{Department of Mathematics\\Northwestern University\\Evanston, IL 60208}
\email{leviborevitz2023@u.northwestern.edu}

\author[Gomes]{Tara Gomes}
\address{School of Mathematics\\University of Minnesota\\Minneapolis, MN 55455}
\email{gomes072@umn.edu}

\author[Ma]{Jiajie Ma}
\address{Department of Mathematics\\Duke University\\Durham, NC 27708}
\email{jm847@math.duke.edu}

\author[Niergarth]{Harper Niergarth}
\address{School of Mathematics\\University of Minnesota\\Minneapolis, MN 55455}
\email{nierg001@umn.edu}

\author[O'Neill]{Christopher O'Neill}
\address{Mathematics and Statistics Department\\San Diego State University\\San Diego, CA 92182}
\email{cdoneill@sdsu.edu}

\author[Pocklington]{Daniel Pocklington}
\address{Department of Mathematics\\Grinnell College\\Grinnell, IA 50112}
\email{pockling@grinnell.edu}

\author[Stolk]{Rosa Stolk}
\address{Department of Mathematics\\Utrecht University\\Utrecht, Netherlands}
\email{rosasofiestolk@gmail.com}

\author[Wang]{Jessica Wang}
\address{Department of Mathematics\\Humboldt University of Berlin\\Berlin, Germany}
\email{jeswang1015@gmail.com}

\author[Xue]{Shuhang Xue}
\address{Department of Mathematics\\University of Virginia\\Charlottesville, VA 22903}
\email{hl6de@virginia.edu}

\date{\today}

\begin{abstract}
A numerical semigroup is a cofinite subset of the non-negative integers that is closed under addition and contains 0.  Each numerical semigroup $S$ with fixed smallest positive element $m$ corresponds to an integer point in a rational polyhedral cone $\Cm$, called the Kunz cone.  Moreover, numerical semigroups corresponding to points in the same face $F \subseteq \Cm$ are known to share many properties, such as the number of minimal generators.  
In~this work, we classify which faces of $\Cm$ contain points corresponding to numerical semigroups.  Additionally, we obtain sharp bounds on the number of minimal generators of $S$ in terms of the dimension of the face of $\Cm$ containing the point corresponding to $S$.  
\end{abstract}

\maketitle


\section{Introduction}
\label{sec:intro}

A \emph{numerical semigroup} is a cofinite subset $S \subseteq \NN$ of the non-negative integers that is closed under addition and contains $0$.  
Of recent interest is a family of rational cones~$\Cm$, called \emph{Kunz cones}, with the property that each numerical semigroup $S$ with smallest positive element $m$ corresponds to an integer point in $\Cm$.  
Inspired by a family of polyhedra introduced by Kunz~\cite{kunz}, these cones have predominently been used to apply lattice point counting techniques (e.g., Ehrhart's theorem) to the enumeration of numerical semigroups~\cite{alhajjarkunz,kunzcoords}, in hopes of making progress on some long-standing open problems in the numerical semigroups literature, such as Wilf's conjecture~\cite{wilfmultiplicity,wilfsurvey,kaplanwilfconj,wilfconjecture} and the Bras-Amoros conjecture concerning the number of numerical semigroups with a given genus~\cite{wilfbruteforce,kaplancounting}.  

More recently, a connection between the faces of $\Cm$ and the numerical semigroups therein was discovered~\cite{wilfmultiplicity,kunzfaces1}.  To each face $F \subseteq \Cm$, a finite poset $P$ can be naturally associated (called the \emph{Kunz poset} of $F$).  If a numerical semigroup $S$ corresponds to a point in $F$ (we often say ``$F$ contains $S$'' for simplicity), then $P$ coincides with a certain subset of the divisibility poset of $S$.  Note that not all faces of $\Cm$ contain numerical semigroups, but every face has an asociated Kunz poset.  
This construction was later enhanced to associate to each $F$ a nilsemigroup $N$, called the \emph{Kunz nilsemigroup} of $F$, in such a way that (i) the Kunz poset $P$ of $F$ is the divisibility poset of the non-nil elements of $N$, and (ii) if a numerical semigroup $S$ lies in $F$, then $N$ is the quotient of $S$ by a certain semigroup congruence~\cite{minprescard,kunzfaces3}.  

An interesting consequence of the above is that the faces of $\Cm$ naturally partition the (infinite) collection of numerical semigroups with smallest positive element $m$ into finitely many equivalence classes with similar algebraic properties.  For instance, if two numerical semigroups lie in the same face of $\Cm$, then they have the same number of minimal generators (it is one more than the number of minimal generators of the Kunz nilsemigroup $N$), and the first Betti numbers of their defining toric ideals coincide~\cite{kunzfaces3}.  In~this way, one may view the Kunz cones as a sort of ``moduli space of numerical semigroups'' though this has yet to be made precise~\cite{modulispaces}.  

In this paper, we answer two natural questions.  

\begin{mainquestion}\label{mq:aperyfaces}
When does a given face $F \subseteq \Cm$ contain numerical semigroups?  
\end{mainquestion}

We answer Main Question~\ref{mq:aperyfaces} in terms of the Kunz nilsemigroup of $F$ (Theorem~\ref{t:main}).  Our characterization is both combinatorial and algebraic in nature, involving both nilsemigroups and the saturation of certain integer lattices, and yields an algorithm for testing whether a given face $F$ contains numerical semigroups.  We also provide some examples that highlight the difficulty in obtaining a purely combinatorial (e.g., poset-theoretic or nilsemigroup-theoretic) characterization.  

One of the applications of Theorem~\ref{t:main} lies in streamlining the process of proving the existence of numerical semigroups with certain properties.  
The central theme of~\cite{rosalesapery,highembdim}, and the recent manuscript~\cite{minprescard}, is locating a numerical semigroup with prescribed smallest positive element $m$, number of minimal generators, and first Betti number.  The latter article~\cite{minprescard} made significant headway on this problem by first constructing a Kunz nilsemigroup $N$, from which all 3 quantities are determined, and then locating a numerical semigroup $S$ lying in the face of $\Cm$ corresponding to $N$.  
Due to forthcoming work of the second and fifth authors of the present manuscript that proves all higher Betti numbers also coincide for numerical semigroups in the same face, and the broader interest these quantities hold (see the survey~\cite{nsbettisurvey}), it is worthwhile to streamline the process of proving there exists a numerical semigroup withh a given Kunz nilsemigroup.  

Presently, after identifying a Kunz nilsemigroup $N$ with the desired properties, one must locate a numerical semigroup in the face corresponding to $N$, the construction of which is often technical and cumbersome due to certain modular arithmetic restrictions that must be satisfied.  
Theorem~\ref{t:main} provides a different (and simpler) avenue:\ after locating the desired Kunz nilsemigroup $N$, prove $N$ corresponds to a face $F \subseteq \Cm$, then prove $F$ contains numerical semigroups.  
We demonstrate the utility of this approach in our answer to the following question.  
Let $\mathsf e(F)$ denote the number of minimal generators of the Kunz nilsemigroup of $F$ (if $F$ contains a numerical semigroup $S$, then $\mathsf e(F)$ is one less than the number of minimal genreators of $S$).  

\begin{mainquestion}\label{mq:edpairs}
Given $d$ and $m$, what are the possible values of $\mathsf e(F)$ for $F \subseteq \Cm$ with $d = \dim F$?  
\end{mainquestion}

It is known~\cite{kunzfaces1} that $\dim F \le \mathsf e(F)$.  
We obtain, for each face $F$ with $\dim F \ge 2$, sharp bounds on $\mathsf e(F)$ in terms of $\dim F$.  We also identify a family of faces for each $m$ that demonstrate every intermediate value of $\mathsf e(F)$ between these bounds is attained.  Our proof of the latter is made more concise and less technical thanks to Theorem~\ref{t:main} (we encourage the reader to compare them to the constructions given in~\cite{minprescard}).  We~also identify when the faces in this family contain numerical semigroups, and for those that do not, we prove that locating such a face that does contain numerical semigroups is impossible (again using Theorem~\ref{t:main}).  

Our answers to Main Questions~\ref{mq:aperyfaces} and~\ref{mq:edpairs} can be found in Sections~\ref{sec:aperyfaces} and~\ref{sec:edpairs}, respectively, after an overview of the background in Section~\ref{sec:background}.

\section{Background}
\label{sec:background}

In this section, we recall Kunz nilsemigroups and presentations (for a thorough introduction, see~\cite{minprescard}), basic definitions from polyhedral geometry (see~\cite{ziegler}), and the Kunz cone (see~\cite{kunzfaces1}).  


Fix a commutative semigroup $(N,+)$.  An element $\infty \in N$ is \emph{nil} if $a+\infty=\infty$ for all $a\in N$. We say an element $a\in N$ is \emph{nilpotent} if $na=\infty$ for some $n\in\ZZ_{\geq 1}$, and \emph{partly cancellative} if $a+b=a+c\neq \infty$ implies $b=c$ for all $b,c\in N.$ We say $N$ is a \emph{nilsemigroup with identity} (or just a \emph{nilsemigroup}) if $N$ has an identity element and every non-identity element is nilpotent, and that $N$ is \emph{partly cancellative} if every non-nil element of $N$ is partly cancellative.  The \emph{atoms} of a partly cancellative semigroup are the elements that cannot be written as a sum of two non-identity elements.  

Fix a partly cancellative semigroup $S$ with finite generating set $n_0, \ldots, n_k$ (we write $S = \langle n_0,\ldots,n_k\rangle$ in this case).  A \emph{factorization} of an element $n \in S$ is an expression
\[
n=z_0n_0+\cdots+z_kn_k
\]
of $n$ as a sum of generators of $S$. The \emph{set of factorizations} of $n \in S$ is the set
\[
\mathsf{Z}_S(n)=\{z\in \ZZ_{\ge 0}^{k+1}:n=z_0n_0+\cdots+z_kn_k\},
\]
viewed as a subset of $\ZZ_{\ge 0}^{k+1}$. 
The \emph{support} of a factorization $z = (z_0,\ldots,z_k)$ is the set $\supp(z) = \{i : z_i > 0\}$ of nonzero coordinates.  
The \emph{factorization homomorphism} of $S$ is the function $\varphi_S:\ZZ_{\ge 0}^{k+1}\to S$ given by 
\[
\varphi_S(z_0,\ldots,z_k)=z_0n_0+\cdots+z_kn_k
\] 
sending each tuple to the element of $S$ it is a factorization of. 
The \emph{kernel} of $\varphi_S$ is 
\[
\ker\varphi_S = \{(z,z')\in \ZZ_{\ge 0}^{k+1} \times \ZZ_{\ge 0}^{k+1}: \varphi_S(z) = \varphi_S(z')\},
\]
which is an equivalence relation $\til$ setting $z \sim z'$ whenever $\varphi_S(z)$ = $\varphi_S(z')$. In fact, $\til$ is a \emph{congruence relation}, meaning that $z \sim z'$ implies $(z + z'') \sim (z' + z'')$ for all $z, z', z'' \in \ZZ_{\ge 0}^{k+1}$ (i.e., $\til$ is closed under \emph{translation}).  Each such relation $z \sim z'$ is called a \emph{trade} of $S$.  A set $\rho$ of trades is said to \emph{generate} $\til$ if $\til$ is the smallest congruence on $\ZZ_{\ge 0}^{k+1}$ containing $\rho$.  A \emph{presentation} of $S$ is a set of trades obtained from a generating set of $\ker\varphi_S$ by omitting any generators $z \sim z'$ for which $\varphi_S(z)$ is nil.

Fix a subset $Z \subseteq \ZZ_{\ge 0}^d$.  The \emph{factorization graph} of $Z$ is the graph $\nabla_Z$ whose vertices are the elements of $Z$ in which two tuples $z, z' \in Z$ are connected by an edge if $z_i > 0$ and $z_i' > 0$ for some $i$.  If $S$ is a finitely generated semigroup and $n \in S$, then we write $\nabla_n = \nabla_Z$ for $Z = \mathsf Z_S(n)$.  The following, which appeared as \cite[Proposition~3.8]{kunzfaces3}, identifies a relationship between factorization graphs and presentations.  

\begin{prop}\label{p:kunzminpresnabla}
Fix a finitely generated, partly cancellative semigroup $S$.  A set $\rho$ of relations between factorizations of elements of $S \setminus \{\infty\}$ is a presentation of $S$ if and only if for every non-nil $a \in S$, a connected graph is obtained from $\nabla_a$ by adding an edge for each pair of factorizations of $a$ in $\rho$.
\end{prop}

A \emph{numerical semigroup} is an additive subsemigroup $S$ of $\ZZ_{\ge 0}$ with finite complement that contains $0$.  Any numerical semigroup has a unique generating set that is minimal with respect to inclusion.  The cardinality of the minimal generating set of $S$ is known as the \emph{embedding dimension} of~$S$, and the smallest positive element of $S$ is called the \emph{multiplicity} of $S$.

Letting $m$ be the multiplicity of $S$, the \emph{Ap\'ery set} of $S$ is the set
\[
\Ap(S)=\{n\in S\,:\, n-m\not\in S\}
\]
of minimal elements of $S$ within each equivalence class modulo $m$. Since $S$ is cofinite, we have that $|\Ap(S)| = m$ and that $\Ap(S)$ contains exactly one element in each equivalence class modulo $m$. If we write 
\[
\Ap(S)=\{0,a_1,\ldots,a_{m-1}\},
\]
where $a_i \equiv i \bmod m$ for each $i=1,\ldots,m-1$, then we refer to the tuple $(a_1,\ldots,a_{m-1})$ as the \emph{Ap\'ery tuple} of $S$.

The \emph{Kunz nilsemigroup} $N$ of a numerical semigroup $S$ is obtained from $S/\sim$, where $\sim$ is the congruence that relates $a \sim b$ whenever $a = b$ or $a,b \notin \Ap(S)$ (the set $S\setminus \Ap(S)$ comprises the nil of $S/\sim$), by replacing each non-nil element with its equivalence class in $\ZZ_m$.  The atoms of $N$ are thus the elements $\ol a \in \ZZ_m$ where $a \ne m$ is a minimal generator of $S$.  

\begin{example}\label{e:kunznilsemigroup}
Consider the numerical semigroup $S = \langle 6,7,8,9\rangle$. Its Ap\'ery set, written as $\{0,a_1,\ldots, a_5\}$ with $a_i \equiv i \bmod 6$, is 
\[
\Ap(S) = \{0, 7, 8, 9, 16, 17\}.
\]
Thus the Kunz nilsemigroup $N$ of $S$ is the set $\{0,1,2,3,4,5,\infty\}$ equipped with the operation $+$ defined so that $0$ is the identity, 
\begin{center}
1 + 3 = 4,
\qquad
2 + 2 = 4,
\qquad
2 + 3 = 5,
\end{center}
and all other sums yield $\infty$.  The divisibility poset of $N \setminus \{\infty\}$, which places $b$ above $a$ whenever $a + c = b$ for some $c \in N$, is depicted in Figure~\ref{f:nilsemigroup-example}. The factorization sets of the non-nil elements in $N$ are 
\begin{alignat*}{3}
\mathsf{Z}_N(0)& = \{(0,0,0)\},\quad \mathsf{Z}_N(1) = \{(1,0,0)\}, \quad &&\mathsf{Z}_N(2) = \{(0,1,0)\}, \\
\mathsf{Z}_N(3)& = \{(0,0,1)\}, \quad \mathsf{Z}_N(4) = \{(1,0,1),(0,2,0)\}, \quad &&\mathsf{Z}_N(5) = \{(0,1,1)\}.
\end{alignat*}
Looking at the factorization graphs of the non-nil elements of $N$, we see that all are connected except $\nabla_4$.  Only one edge is needed to connect it, namely the edge corresponding to the trade $(1,0,1) \sim (0,2,0)$.  Therefore, $\rho = \{((1,0,1),(0,2,0))\}$ is a presentation for $N$ by Proposition~\ref{p:kunzminpresnabla}.
\end{example}

\begin{figure}[t!]
\centering
\includegraphics[scale=0.4]{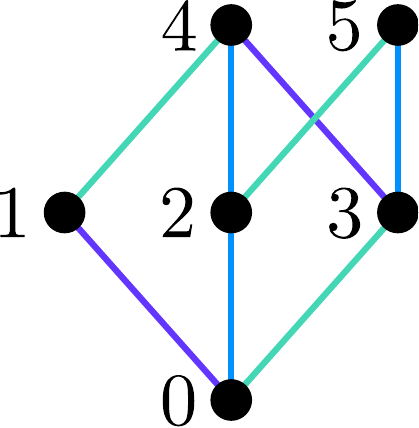}
\caption{The divisibility poset of the non-nil elements of the Kunz nilsemigroup of $S = \langle 6, 7, 8, 9\rangle$.}
\label{f:nilsemigroup-example}
\end{figure}

A \emph{rational polyhedral cone} $C \subseteq \RR^d$ is the set of solutions to a finite list of linear inequalities with rational coefficients, that is,
\[
P=\{ x\in \RR^d:Ax \ge 0 \}
\] 
for some rational matrix $A$.  We say $C$ is \emph{pointed} if it does not contain any positive-dimensional linear subspaces of $\RR^d$, and the \emph{dimension} of $C$ is the vector space dimension of $\spann_\RR C$.  
If none of the inequalities can be omitted without altering $C$, we call this list an \emph{$H$-description} or \emph{facet description} of $C$ (if $C$ is full-dimensional, then such a list of inequalities is unique up to reordering and scaling by positive constants).  The inequalities appearing in the $H$-description of $C$ are called \emph{facet inequalities} of $C$.

Given a facet inequality $a_1x_1+\cdots a_dx_d \ge 0$ of $C$, the intersection of $C$ with the equation $a_1x_1+\cdots a_dx_d = 0$ is called a \emph{facet} of $C$.  Any intersection $F$ of facets of $C$ is called a \emph{face} of $C$, and it itself a rational polyhedral cone.  The set of facets containing a face $F$ is called the \emph{facet description} of $F$.  The \emph{interior} of $F$, denoted $F^\circ$, is the set of points in $F$ that do not lie in any faces properly contained in $F$.  

Fix $m \ge 2$. The \emph{Kunz cone} $\Cm \subseteq \RR_{\geq 0}^{m-1}$ is the pointed cone with facet inequalities
\[
x_i+x_j\geq x_{i+j}
\qquad
\text{for}
\qquad
i,j \in \ZZ_m\setminus \{0\}
\qquad
\text{with}
\qquad
i+j\not = 0,
\] 
where the coordinates of $\RR^{m-1}$ are indexed by $\ZZ_m \setminus \{0\}.$ 
The integer points $(z_1,\ldots,z_{m-1})$ satisfying $z_i \equiv i \bmod m$ for each $i$ are called \emph{Ap\'ery points}.  

\begin{thm}[{\cite{kunz}}]\label{t:kunzcoords}
For each $m \ge 2$, the set of Ap\'ery points in $\Cm$ coincides with the set of Ap\'ery tuples of numerical semigroups containing $m$.
\end{thm}

We say a face $F \subseteq \Cm$ is \emph{degenerate} if there exists a point $x \in F^\circ$ with some $x_i = 0$.  Positive dimensional degenerate faces occur in $\Cm$ whenever $m$ is composite (see \cite[Example~3.6]{kunzfaces1}).  However, such faces never contain Ap\'ery points.  The following result, which appeared as \cite[Theorem~3.3]{kunzfaces3}, implies that the Kunz nilsemigroup of a numerical semigroup $S$ uniquely identifies the face of $\Cm$ containing its Ap\'ery tuple.  

\begin{thm}\label{t:kunznilsemigroup}
Fix a non-degenerate face $F \subseteq \Cm$. Define a commutative operation $\oplus$ on the set $\ZZ_m \cup \{\infty\}$ so that $\infty$ is nil and for all $a,b \in \ZZ_m$,
\[
a \oplus b = 
\begin{cases}
a+b & \text{if } x_a + x_b = x_{a+b}; \\
\infty & \text{otherwise.}
\end{cases}
\]
\begin{enumerate}[(a)]
\item 
Under this operation, $(N,\oplus)$ is a partly cancellative nilsemigroup (which we call the \emph{Kunz nilsemigroup of $F$}).

\item 
When $F$ contains a numerical semigroup $S$, the Kunz nilsemigroup of $S$ equals the Kunz nilsemigroup of $F$.

\end{enumerate}
\end{thm}

In view of Theorems~\ref{t:kunzcoords} and~\ref{t:kunznilsemigroup}, given a numerical semigroup $S$ and a face $F \subseteq \Cm$, we write $S \in F$ if the Ap\'ery tuple of $S$ lies in $F^\circ$.  

\begin{remark}\label{r:degenerate}
The degenerate faces are precisely those in the image of the injection given by \cite[Corollary~3.7]{kunzfaces1}, hence it is sufficient to classify the embedding and face dimension of non-degenerate faces.  For these reasons, from a classification standpoint, it suffices to consider non-degenerate faces of $\Cm$.  
\end{remark}

\begin{example}\label{e:kunzcoords}
Let us return to the numerical semigroup $S = \langle 6,7,8,9 \rangle$ from Example~\ref{e:kunznilsemigroup}.  Since the Kunz nilsemigroup of $S$ is equal to the Kunz nilsemigroup of the face $F \subseteq \mathcal C_6$ it lies on, the facet equalities of $F$ are 
\begin{align*}
x_1 + x_3 = x_4,
\qquad
x_2 + x_2 = x_4,
\qquad
\text{and}
\qquad
x_2 + x_3 = x_5.
\end{align*}
Furthermore, the Ap\'ery tuple of $S$, $(7,8,9,16,17)$, satisfies the above facet equalities, indicating that these are in fact the $3$ facets of $\mathcal{C}(\ZZ_6)$ containing $(7,8,9,16,17)$.  
\end{example}

Fix a non-degenerate face $F \subseteq \Cm$ with Kunz nilsemigroup $N$ and contained in $i$ facets.  The \emph{embedding dimension} of $F$ is the number $\mathsf e(F) = k$ of atoms of $N$.  If $F$ contains a numerical semigroup~$S$, then $\mathsf e(S) = \mathsf e(F) + 1$.  
The \emph{hyperplane matrix} of $F$ is the matrix $H_F\in\ZZ^{i\times (m-1)}$ whose columns are indexed by the nonzero elements of $\ZZ_m$ and whose rows encode the equations of the facets containing $F$.  Note $\dim F = \rk H_F$.  
Given any finite presentation $\rho$ of $N$, the matrix $M_\rho\in \ZZ^{|\rho|\times k}$ whose columns are indexed by the atoms of $N$ and whose rows have the form $z-z'$ for $(z,z') \in \rho$ is called a \emph{presentation matrix} of $N$.
The following appeared as \cite[Theorem~4.3]{kunzfaces3} and expresses $\dim F$ in terms of $\rk M_\rho$.  

\begin{thm}\label{t:dimension}
If $F \subseteq \Cm$ is a non-degenerate face with Kunz nilsemigroup $N$, then 
\[
\dim F = \mathsf e(F) - \rk(M_\rho)
\]
where $\rho$ is any finite presentation of $N$. 
\end{thm}

\section{A criterion for Ap\'ery faces}
\label{sec:aperyfaces}

Not all non-degenerate faces of the Kunz cone $\Cm$ contain Ap\'ery points.  In this section, we provide a criterion for a face $F \subset \Cm$ to contain Ap\'ery points (Theorem~\ref{t:main}). We also discuss some consequences, including an algorithm to determine whether a face contains Ap\'ery points.

\begin{defn}\label{d:aperyface}
Fix $m \ge 2$ and a face $F \subseteq \Cm$ with Kunz nilsemigroup $N$.  We say $F$~is \emph{Ap\'ery}, and that $N$ is \emph{Ap\'ery}, if $F^\circ$ contains an Ap\'ery point.  
\end{defn}

\begin{example}\label{e:nonaperyex}
Consider the face $F, F' \subseteq \mathcal{C}_6$ whose Kunz nilsemigroups $N$ and $N'$ have divisibility posets depicted in Figures~\ref{f:nonaperyex1} and~\ref{f:nonaperyex2}, respectively.  If $N$ were the Kunz nilsemigroup of a numerical semigroup $S$ with Ap\'ery set $\{0,a_1,\ldots,a_5\}$, then $2a_1 = a_2 = 2a_4$, which is impossible since $a_1$ and $a_4$ are distinct modulo $6$.  Similarly, if $N'$ were the Kunz nilsemigroup of~$S$, then since $a_2=2a_1=a_3+a_5$ we either have $a_3<a_1<a_5$ or $a_5<a_1<a_3$.  Either way, it is impossible to have $a_4=2a_5=a_1+a_3$. Hence, neither $F$ nor $F'$ are Ap\'ery.
\end{example}

\begin{figure}[t!]
\begin{center}
\begin{subfigure}[t]{0.20\textwidth}
\begin{center}
\includegraphics[scale=0.35]{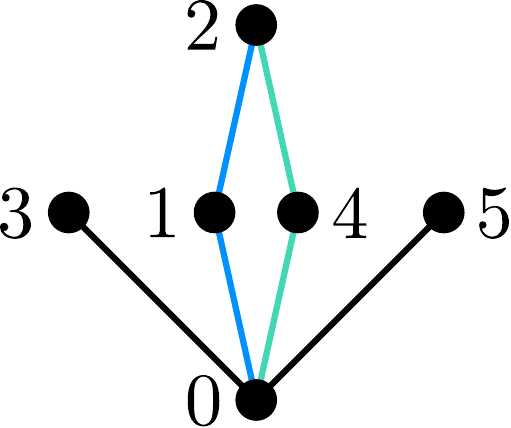}
\end{center}
\caption{}
\label{f:nonaperyex1}
\end{subfigure}
\hspace{0.02\textwidth}
\begin{subfigure}[t]{0.20\textwidth}
\begin{center}
\includegraphics[scale=0.35]{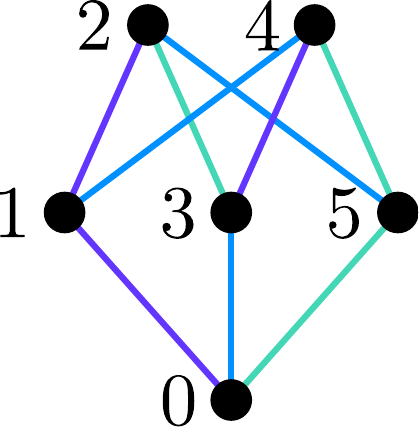}
\end{center}
\caption{}
\label{f:nonaperyex2}
\end{subfigure}
\hspace{0.02\textwidth}
\begin{subfigure}[t]{0.20\textwidth}
\begin{center}
\includegraphics[scale=0.35]{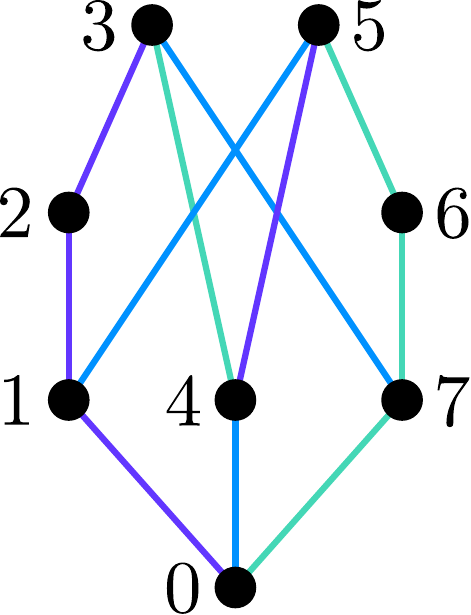}
\end{center}
\caption{}
\label{f:nonaperyex3}
\end{subfigure}
\hspace{0.02\textwidth}
\begin{subfigure}[t]{0.26\textwidth}
\begin{center}
\includegraphics[scale=0.35]{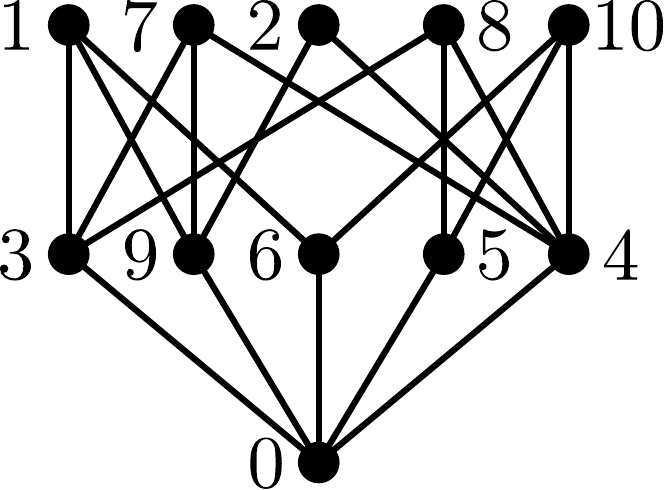}
\end{center}
\caption{}
\label{f:nonaperyex4}
\end{subfigure}
\end{center}
\caption{Divisibility posets of non-Ap\'ery Kunz nilsemigroups.}
\label{fig-examples}
\end{figure}


We briefly recall some definitions (see~\cite[Chapter~7]{cca}).  A \emph{lattice} is a $\ZZ$-submodule $L \subseteq \ZZ^n$, and the \emph{rank} of $L$ is $\rk(L) = \dim_\QQ \spann_\QQ L$.  The \emph{saturation} of $L$ is 
\[
\Sat(L) \coloneq \spann_\QQ(L) \cap \ZZ^n,
\]
and we say $L$ is \emph{saturated} if $L = \Sat(L)$.  The \emph{dual} or \emph{orthogonal complement} of $L$ is 
\[
L^\perp = \{v \in \ZZ^n : v \cdot v' = 0 \text{ for all } v' \in L\},
\]
which is a saturated lattice, even if $L$ is not saturated, and in particular $L^{\perp\perp} = \Sat(L)$.  

\begin{defn}\label{d:tradelattice}
Fix a Kunz nilsemigroup $N$ and a presentation $\rho$ of $N$. The \emph{presentation lattice} of $N$ is given by
$$L_N = L_\rho = \{ c_1 (z_1 - z_1') + \cdots + c_n (z_n - z_n') : c_1, \ldots, c_n \in \ZZ, (z_1,z_1'), \ldots, (z_n, z_n') \in \rho\}.$$
\end{defn}

The following result ensures the presentation lattice $L_\rho$ does not depend on the chosen presentation $\rho$ of $N$, justifying the alternative notation $L_N$.  

\begin{prop}\label{p:thelattice}
If $\rho, \rho'$ are presentations of a Kunz nilsemigroup $N$, then $L_\rho=L_{\rho'}$.
\end{prop}

\begin{proof}
Let $(z,z') \in \rho$. Since $\rho$ and $\rho'$ generate the same congruence, $(z,z')$ can be obtained via translation, transitivity, and symmetry from trades in $\rho'$. Thus, $z-z'$ is an integral linear combination of vectors in $L_{\rho'}$, ensuring $z-z' \in L_{\rho'}$.  This means $L_\rho \subseteq L_{\rho'}$, and so by an identical argument, $L_\rho = L_{\rho'}$.
\end{proof}

We are now ready to state the main theorem of this section, though we defer its proof to later in the section.  

\begin{thm}\label{t:main}
Fix a non-degenerate face $F \subseteq \Cm$ with Kunz nilsemigroup $N$. Let $a_1, \ldots, a_k$ denote the atoms of $N$, and let $\alpha = (a_1, \ldots, a_k) \in \ZZ^e$.  Then $F$ is Ap\'ery if and only if $v \cdot \alpha \equiv 0 \bmod m$ for all $v \in \Sat(L_N)$. 
\end{thm}

\begin{example}\label{e:main}
Consider face $F \subseteq \mathcal C_8$ whose Kunz nilsemigroup $N$ has divisibility poset depicted in Figure~\ref{f:nonaperyex3}.  Using the presentation 
$$\rho = \{((3,0,0),(0,1,1)),((0,0,3),(1,1,0)\}$$
of $N$, one can argue directly that $F$ is not Ap\'ery since any Ap\'ery point $(a_1, \ldots, a_7) \in F$ must satisfy $4a_1 = a_1 + a_2 + a_3 = 4a_3$, which is impossible since $a_1 \not\equiv a_3 \bmod 8$.  Theorem~\ref{t:main} ensures this same conclusion since 
$$(3,-1,-1)-(-1,-1,3)=(4,0,-4)\in L_N,$$
implies $(1,0,-1) \in \Sat(L_N)$, but $(1,0,-1) \cdot (1,4,7)=-6 \not\equiv 0 \bmod 8$.  
\end{example}

Most ``small'' examples of non-Ap\'ery faces, including all 3 thus far in this section, rely on $m$ being composite, in that $ax \equiv b \bmod m$ has 2 or more solutions when $a \ge 2$ divides $m$.  However, $\Cm$ has non-Ap\'ery faces for each $m \ge 11$.  

\begin{example}\label{e:primemult}
The face $F \subseteq \mathcal C_{11}$, whose Kunz nilsemigroup $N$ has divisibility poset depicted in Figure~\ref{f:nonaperyex4}, is not Ap\'ery.  Indeed, one presentation of $N$ is
\begin{align*}
    \rho&=\{((1,0,0,0,1),(0,2,0,0,0)),((1,1,0,0,0),(0,0,2,0,0)), \\
    &\quad\quad((0,0,1,0,1),(0,0,0,2,0)),((1,0,0,1,0),(0,0,0,0,2))\},
\end{align*}
and since 
$$3 (1,-2,0,0,1) + 6(1,1,-2,0,0) + (0,0,1,-2,1) + 2(1,0,0,1,-2) = (11,0,-11,0,0),$$
we must have $(1,0,-1,0,0) \in \Sat(L_N)$.  Theorem~\ref{t:main} then implies $F$ is not Ap\'ery.
\end{example}

Before proving Theorem \ref{t:main}, we develop some necessary machinery.  
We begin with Proposition~\ref{p:hyperplanes}, which ensures that in order to locate an Ap\'ery point in the interior of a face $F \subseteq \Cm$, it is equivalent to locate an Ap\'ery point on $\spann_\QQ F$. 

\begin{prop}\label{p:hyperplanes}
Fix a face $F \subseteq \Cm$, and let $V = \spann_\QQ F$.  If there exists an Ap\'ery point in $V$, then $F$ is an Ap\'ery face.  
\end{prop}

\begin{proof}
Let $d = \dim F$, and assume there exists an Ap\'ery point $x \in V$. 
Fix rational vectors $v_1,\ldots,v_d \in F$ that form a basis for $V$.  By scaling appropriately, we may assume $v_1, \ldots, v_d$ have entries that are integer multiples of $m$.  Writing $x = c_1 v_1 + \cdots + c_d v_d$ where each $c_i \in \QQ$, the point
$$x' = x + \sum_{i = 1}^d  \max(\lceil -c_i \rceil, 1) v_i$$
is an Ap\'ery point since $x_i' \equiv x_i \bmod m$ for each $i$, and $x' \in F^\circ$ since it is a positive linear combination of $v_1, \ldots, v_d$.  This completes the proof.  
\end{proof}



Our next step is Theorem~\ref{t:general}, a general result we will apply to determine the existence of Ap\'ery points in the lattice $L = (\spann_\QQ F) \cap \ZZ^{m-1}$ using vectors in the orthogonal lattice $L^\perp$.  
We first prove a technical lemma.  


\begin{lemma}\label{l:general}
If $L \subseteq \ZZ^n$ is a saturated lattice of rank $d$, and $\overline L \subseteq \ZZ_m^n$ is its image under the canonical projection $\ZZ^n \to \ZZ_m^n$, then $\overline L$ is a free $\ZZ_m$-module of rank $d$ and $\ol{L^\perp} = \ol{L}^\perp$.  
\end{lemma}

\begin{proof}
Let $v_1,\ldots,v_d$ be a $\ZZ$-basis for $L$.  Suppose that there exist $c_i \in \ZZ$ such that 
\[
\sum_{i=1}^d \overline c_i \overline v_i = \overline 0,
\qquad \text{that is,} \qquad
\sum_{i=1}^d c_iv_i = mz
\]
for some $z \in \ZZ^n$.  Since $L$ is saturated, $z \in L$, and since the $v_i$'s form a $\ZZ$-basis for $L$, $z$~is uniquely expressed as
\[
z=\sum_{i=1}^d \frac{c_i}{m}v_i,
\]
where the coefficients $\tfrac{c_i}{m}$ are integers.  Hence $m \mid c_i$ for each $i$, and we conclude $\ol v_1, \ldots, \ol v_d$ form a basis for $\ol L$.  This means $\ol L$ is a free $\ZZ_m$-module of rank $d$.  

For the final claim, we induct on $d$.  If $d=0$, then $\ol{L}^\perp = \ZZ_m^n = \ol{L^\perp}$ for every $m \ge 2$, so suppose $d > 0$.  Let $v_1,\ldots,v_d$ be a $\ZZ$-basis for $L$, let
\[
T = \spann_\ZZ\{v_1,\ldots,v_{d-1}\},
\]
and assume $\ol{T^\perp} = \ol{T}^\perp$ holds for every $m \ge 2$.  Fix a $\ZZ$-basis $w_1,\ldots,w_{n-d}$ for $L^\perp$.  Since $L^\perp$ and $T^\perp$ are saturated, $T^\perp/L^\perp \cong \ZZ$, so choosing any $w' \in T^\perp$ whose image is $1 \in \ZZ$ yields a $\ZZ$-basis $w_1,\ldots, w_{n-d}, w'$ is for $T^\perp$.  

Since $w' \notin L^\perp$, $w' \cdot v_d \neq 0$.  We claim $|w' \cdot v_d| = 1$.  For the sake of contradiction, suppose $w' \cdot v_d = \ell > 1$.  Applying the inductive hypothesis for $m = \ell$, we~have $\ol{L}^\perp \subsetneq \ol{T}^\perp = \ol{T^\perp}$, where the strict containment follows from $\ol T \subsetneq \ol L$ (they are free $\ZZ_\ell$-modules of distinct ranks) and $\ol L^{\perp\perp} = \ol L$ (by \cite[Theorem~4.1 and Corollary~4.5]{WJK13}).  However, $w_i \cdot v_j = 0$ for all $i$ and~$j$, $w' \cdot v_j = 0$ for $j < d$, and $w' \cdot v_d \equiv 0 \bmod \ell$, so $\overline{L}^\perp = \overline{T^\perp}$, which is a contradiction.  

Now, fix $m \ge 2$, and consider reduction modulo $m$.  We have $\ol{L^\perp} \subseteq \ol{L}^\perp \subsetneq \ol{T^\perp}$, and $k\ol{w'} \notin \ol{L}^\perp$ for $1 \le k < m$ since the claim in the preceding paragraph implies $kw' \cdot v_d = k \not\equiv 0 \bmod m$.  Thus, the order of $\overline{w'}$ in $\ol{T^\perp}/ \ol{L}^\perp$ is $m$, so $|\ol{T^\perp}/ \ol{L}^\perp| \geq m$. By~Lagrange's theorem, $|\ol{L}^\perp| = m^d = |\ol{L^\perp}|$, thereby completing the proof.  
\end{proof}

\begin{example}\label{counterexamples}
Both parts of Lemma~\ref{l:general} can fail if the lattice $L$ is not saturated.  Indeed, the image of the lattice $L \subseteq \ZZ^2$ generated by $(2,0)$ under the projection into $\ZZ_4^2$ is $\ol L = \{(\ol 0,\ol 0),(\ol 2,\ol 0)\}$, which has no basis since $\ol 2(\ol 2,\ol 0) = (\ol 0,\ol 0)$.  Moreover, $L^\perp = \spann_\ZZ\{(0,1)\}$ projects to $\ol{L^\perp} = \{(\ol 0,\ol 0),(\ol 0,\ol 1),(\ol 0,\ol 2),(\ol 0,\ol 3)\}$, so $(\ol 2,\ol 0) \in \ol{L}^\perp \setminus \ol{L^\perp}$.   
\end{example}

\begin{thm}\label{t:general}
Fix a saturated lattice $L \subseteq \ZZ^n$, $\gamma \in \ZZ^n$, and $m \in \ZZ_{\ge 2}$.  There exists $v \in L$ with $v_i \equiv \gamma_i \bmod m$ for each $i$ if and only if $w \cdot \gamma \equiv 0 \bmod m$ for every $w \in L^\perp$. 
\end{thm}

\begin{proof}
Lemma~\ref{l:general} and \cite[Theorem~4.1 and Corollary~4.5]{WJK13} imply $\ol{L^\perp}^\perp = \overline{L}^{\perp\perp} = \ol L$.  To~prove the claim, it suffices to prove $\ol \gamma \in \ol L$ if and only if $\ol \gamma \in \ol{L^\perp}^\perp$, so we are done.  
\end{proof}

Letting $L = \ker(H_F) \cap \ZZ^{m-1}$ and $\gamma = (1,2,\ldots,m-1)$, a consequence of Theorem~\ref{t:general} is that $F$ is Ap\'ery if and only if $w \cdot \alpha \equiv 0 \bmod m$ for every $w \in \Row(H_F) \cap \ZZ^{m-1}$.  But Theorem~\ref{t:main} states it is enough to consider vectors in the lower-rank lattice $\Sat(L_N)$.  
To this end, we recall Theorem~\ref{t:kunz3}, a consequence of the proof of \cite[Theorem~4.3]{kunzfaces3} that relates $H_F$ to a presentation matrix of the Kunz nilsemigroup of $F$, as the final ingredient in our proof of Theorem~\ref{t:main}.  

\begin{thm}\label{t:kunz3}
Fix a non-degenerate face $F \subseteq \Cm$ with Kunz nilsemigroup $N$, and order the non-nil elements $p_1 \preceq \cdots \preceq p_{m-1}$ of $N$ so that $p_1,\ldots,p_k$ are the atoms of~$N$. Then there exist invertible matrices $R$ and $C$ such that
\[
RH_FC = 
\left(\begin{array}{r|r}
M_\rho & 0 \\
\\[-0.9em]
\hline
\\[-0.9em]
-A & I
\end{array}\right)
\]
where the columns on the right hand side are labeled by $p_1, \ldots, p_{m-1}$, $M_\rho$ is a presentation matrix of $N$, $I$ is an identity matrix, and $A$ is a matrix whose $i$-th row is a factorization of $p_{k+i}$ for each $i \leq m-k-1$.
\end{thm}

\begin{proof}[Proof of Theorem \ref{t:main}]
Fix a multiplicity $m$ and a non-degenerate face $F \subseteq \Cm$ with Kunz nilsemigroup $N$ and hyperplane matrix $H_F$. Applying Theorem~\ref{t:general} with
\[
L = \ker(H_F) \cap \ZZ^{m-1}
\qquad \text{and} \qquad
\gamma = (1,2,\ldots,m-1),
\] 
we have that $F$ is Ap\'ery if and only if $w \cdot \gamma \equiv 0 \bmod{m}$ for all $w \in \Row(H_F) \cap \ZZ^{m-1}$.  
It is now sufficient to show the statement
\begin{equation}\label{eq:maina}
w \cdot \gamma \equiv 0 \bmod{m}
\quad \text{for all} \quad
w \in \Row(H_F) \cap \ZZ^{m-1}
\end{equation}
is equivalent to 
\begin{equation}\label{eq:mainb}
w \cdot \alpha \equiv 0 \bmod{m}
\quad \text{for all} \quad
w \in \Sat(L_N).
\end{equation}
Letting $\widetilde{H}_F = R H_F C$ as in Theorem~\ref{t:kunz3} and $\widetilde{\gamma} = \gamma C$, \eqref{eq:maina} is equivalent to 
\begin{equation}\label{eq:mainc}
w \cdot\widetilde{\gamma}\equiv 0\bmod m
\quad \text{for all} \quad
w \in \Row(\widetilde{H}_F) \cap \ZZ^{m-1}
\end{equation}
since $\Row(\widetilde{H}_F) = \Row(H_FC)$.  

Fix $w = c\widetilde{H}_F \in \Row(\widetilde{H}_F) \cap \ZZ^{m-1}$ with $c \in \QQ^{m-1}$.  
Due to the form of~$\widetilde{H}_F$, we can write $c = (c', c'')$ so that $w = (c' M_\rho, 0) + (-c'' A, c'' I)$.  Since each coordinate of $c''$ coincides with a coordinate of $w$, each $c_i'' \in \ZZ$, so by the definition of $A$, we have $(-c'' A, c'' I) \cdot \widetilde \gamma \equiv 0 \bmod m$.  As such, $w \cdot \widetilde \gamma \equiv 0 \bmod m$ if and only if 
\[
(c' M_\rho, 0) \cdot \widetilde \gamma = (c' M_\rho) \cdot \alpha \equiv 0 \bmod m,
\]
which proves~\eqref{eq:mainc} is equivalent to~\eqref{eq:mainb} since $\Row(M_\rho) \cap \ZZ^k = \Sat(L_N)$.  
\end{proof}

\begin{cor}\label{c:aperylattices}
If a non-degenerate face $F \subset \Cm$ has Kunz nilsemigroup $N$ satisfying any of the following conditions, then $F$ is Ap\'ery:
\begin{enumerate}[(a)]
\item \label{eq:aperylatticesa}
$N$ has a saturated presentation lattice;

\item \label{eq:aperylatticesb}
$N$ has a trivial presentation lattice; or

\item \label{eq:aperylatticesc}
$F$ is a facet.

\end{enumerate}
\end{cor}

\begin{proof}
Assume $F$ satisfies~\eqref{eq:aperylatticesa}.  Letting $\alpha$ the vector of atoms of $N$, every $\tau \in L_N$ satisfies $\tau \cdot \alpha \equiv 0 \bmod m$ by definition.  However, by assumption, $L_N = \Sat(L_N)$, so $F$ is Ap\'ery by Theorem~\ref{t:main}.  Since~\eqref{eq:aperylatticesb} implies~\eqref{eq:aperylatticesa}, this proves the first two claims.  

To complete the proof, we claim~\eqref{eq:aperylatticesc} implies~\eqref{eq:aperylatticesb}.  Indeed, if $N$ has a non-trivial presentation $\rho$, then $N$ has at most $m-2$ atoms and $\rk(M_\rho) \ge 1$.  Hence, Theorem~\ref{t:dimension} implies $\dim F \le m - 3 < \dim \Cm - 1$, so $F$ is not a facet.
\end{proof}

\begin{example}\label{e:saturatedlatticecounterexample}
Given Corollary~\ref{c:aperylattices}, one might hope that Ap\'ery faces are precisely those with saturated presentation lattice.  Unfortunately, this is not true in general.  The Kunz nilsemigroup $N$ of the numerical semigroup
\[
S = \langle 18, 41, 43, 83, 85, 92, 96, 99, 106 \rangle,
\]
seen in Figure~\ref{f:saturatedlatticecounterexample}, has presentation $\rho=\{(0, 3, 0, 1, 0, -2, 0, 0),(0, 1, 0, 3, 0, 0, -2, 0)\}$.  One can check that
\[
\tfrac{1}{2}(0, 3, 0, 1, 0, -2, 0, 0)+ \tfrac{1}{2}(0, 1, 0, 3, 0, 0, -2, 0)=(0,2,0,2,0,-1,-1,0) \in \Sat(L_N) \setminus L_N.
\]
Despite this, the face containing $S$ does not violate Theorem~\ref{t:main} since the right hand side above has the desired dot product with the vector of atoms of $N$.  
\end{example}

\begin{figure}[t!]
\centering
\includegraphics[scale=0.3]{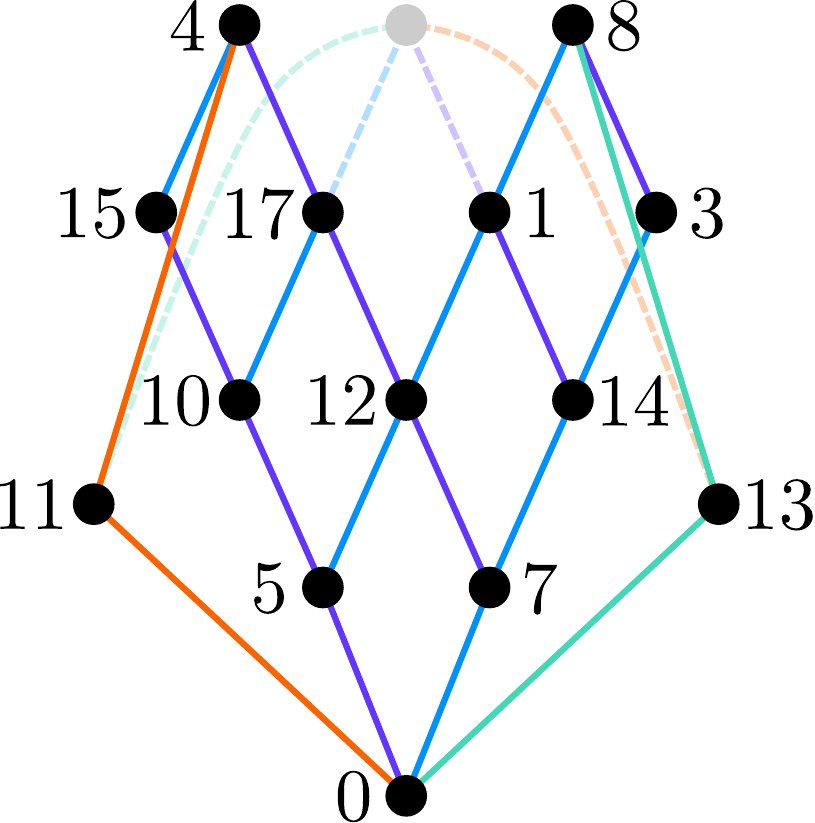}
\caption{Divisibility poset of an Ap\'ery Kunz nilsemigroup whose presentation lattice $L_\rho$ is not saturated.  Atoms 2, 6, 9, and 16 are omitted for clarity, and the opaque vertex and dashed edges identify the location of a trade in $\Sat(L_\rho) \setminus L_\rho$.}
\label{f:saturatedlatticecounterexample}
\end{figure}



One of the primary applications of Theorem~\ref{t:main} is an algorithm for determining whether a face $F$ is Ap\'ery using its Kunz nilsemigroup $N$ (Algorithm~\ref{a:aperycheck}).  
Table~\ref{tb:aperydata} gives some data on the distribution of Ap\'ery faces for $m \le 13$, computed using an implementation of Algorithm~\ref{a:aperycheck} in \texttt{SageMath} dependent on the packages \texttt{numsgps-sage} and \texttt{kunzpolyhedron}~\cite{numericalsgpssage,kunzposetsage}, as well as the \texttt{GAP} package \texttt{numericalsgps}~\cite{numericalsgpsgap}.

\begin{alg}\label{a:aperycheck}
Check if a given Kunz nilsemigroup $N$ is Ap\'ery.  
\begin{algorithmic}
\Function{IsApery}{$N$}
\State $\alpha \gets (a_1, \ldots, a_e)$ atoms of $N$
\State $\rho \gets$ a presentation of $N$
\State $B \gets$ a basis of $\Sat(L_\rho)$
\If{$v\cdot\alpha = 0$ for each $v \in B$}
\State \Return True
\Else
\State \Return False
\EndIf

\EndFunction
\end{algorithmic}
\end{alg}






We close the section with one final result that will be useful in the next section.  

\begin{cor}\label{cor-congruence-problem}
Fix a non-degenerate face $F \subseteq \Cm$ with Kunz nilsemigroup $N$.  If~$n(e_i - e_j) \in L_N$ for some $n \in \ZZ_{\ge 2}$ and $i \neq j$, then $F$ is not Ap\'ery. 
\end{cor}

\begin{proof}
Since $n(e_i - e_j) \in L_N$, $e_i - e_j\in \Sat(L_N)$. Letting $\alpha = (a_1,\ldots,a_k)$ be the vector of atoms of $N$, we have $(e_i-e_j) \cdot \alpha \equiv a_i-a_j \bmod m$, which is nonzero since $a_i$ and $a_j$ are distinct modulo $m$. Thus, $F$ is not Ap\'ery by Theorem~\ref{t:main}.
\end{proof}

\begin{table}
\centering
\begin{tabular}{l|*{6}{@{\,\,}l@{\,}r@{\,\,}|}@{\,\,}l@{\!}r@{\,\,}|}
\multicolumn{1}{c}{ } & \multicolumn{2}{c}{\!\!\!\!$m=6$} & \multicolumn{2}{c}{\!\!\!$m=8$} & \multicolumn{2}{c}{\!\!\!\!$m=9$} & \multicolumn{2}{c}{\!\!\!$m=10$} & \multicolumn{2}{c}{\!\!\!$m=11$} & \multicolumn{2}{c}{\!\!\!$m=12$} & \multicolumn{2}{c}{\!\!$m=13$} \\
$d$ & \ \# & \% & \ \# & \% & \ \# & \% & \ \# & \% & \ \# & \% & \ \# & \% & \ \# & \%
\\
\hline
1 & 6(2) & 75 & 23(2) & 54 & 48(0)   & 40 & 168(16)  & 78 & 142(70) & 30 & 1398(6)    & 84 & 1257(72) & 18 \\
2 & 6(0) & 21 & 83(0) & 34 & 111(13) & 14 & 909(8)   & 46 & 60(60)  &  1 & 12773(24) & 63 & 3681(852) &  4 \\
3 & 2(2) & 7  & 93(2) & 20 & 68(0)   &  4 & 1764(26) & 29 &         &    & 36213(123) & 43 & 708(780)  &  0 \\
4 &      &    & 34(0) &  9 & 12(12)  &  1 & 1510(0)  & 17 &         &    & 48354(24)  & 28 &         &    \\
5 &      &    & 3(3)  &  2 &         &    & 586(4)   &  9 &         &    & 34149(122) & 17 &         &    \\
6 &      &    &       &    &         &    & 92(0)    &  4 &         &    & 12871(0)   & 10 &         &    \\
7 &      &    &       &    &         &    & 4(4)     &  1 &         &    & 2424(22)   &  5 &         &    \\
8 &      &    &       &    &         &    &          &    &         &    & 196(0)     &  2 &         &    \\
9 &      &    &       &    &         &    &          &    &         &    & 5(5)       &  1 &         &    \\
\end{tabular}
\caption{For each $m$, the number of dimension $d$ non-Ap\'ery faces in~$\Cm$ (and how many are maximal under containment among non-Ap\'ery faces), along with the percentage of $d$-dimensional faces that are non-Ap\'ery.}
\label{tb:aperydata}
\end{table}

\section{Face dimension and embedding dimension}
\label{sec:edpairs}

In this section, we locate in Theorem~\ref{t:edimbounds} the extremal values of $\mathsf e(F)$ that can be attained for faces $F \subseteq \Cm$ for fixed $m$ and $\dim(F) \ge 2$, and identify in Proposition~\ref{p:nonapery} for which such values there are no Ap\'ery faces.  We also construct in Theorem~\ref{t:attainable} a family of Kunz nilsemigroups that attains those values and many of the values between them, demonstrating that for $\dim(F) \ge 2$, the set of attainable values is an interval.  

\begin{lemma}\label{l:lemmadimbound}
If $N$ is the Kunz nilsemigroup of a non-degenerate face $F \subseteq \Cm$, then 
\[
\dim(F) \geq e(F) - \!\!\!\!\! \sum_{p \in N - \{\infty\}} \!\!\!\! (|\mathsf{Z}_N(p)| - 1).
\]
\end{lemma}

\begin{proof}
For every $p\in N\setminus \{\infty\}$, the factorization graph $\nabla_p$ has $|\mathsf{Z}_N(p)|$ vertices, so it suffices to add at most $|\mathsf{Z}_N(p)|-1$ edges for $\nabla_p$ to be connected. By Proposition~\ref{p:kunzminpresnabla}, there is a presentation $\rho$ such that $|\rho|\leq \sum_{p \in N - \{\infty\}}(|\mathsf{Z}_N(p)|-1)$. Since $\rk(M_\rho)\leq |\rho|$, an application of Theorem~\ref{t:dimension} gives us the desired result.
\end{proof}

\begin{lemma} \label{l:coordsum}
Let $N$ be a Kunz nilsemigroup, $m = \mathsf m(N)$, and $e = \mathsf e(N)$.  If $p \in N$ is non-nil and $A \subseteq \mathsf Z_N(p)$ is the set of factorizations of $p$ with coordinate sum $2$, then:
\begin{enumerate}[(a)]
\item $|A| \le \lfloor \tfrac{1}{2}e\rfloor + 1$;
\item if $m$ is odd and $e$ is even, $|A|\leq \tfrac{1}{2}e$; and
\item if $m$ is even and $|A| = \lfloor \tfrac{1}{2}e\rfloor+1$, then $p = 2q$ for some $q \in N$.
\end{enumerate}
\end{lemma}

\begin{proof}

We first prove part~(b).  Since $p = 2q$ has a unique solution for $q \in \ZZ_m$ when $m$ is odd, $p$ has at most one singleton-support factorization.  Further, all factorizations in $A$ must have disjoint support by partial cancellativity of $N$.  If $p$ has no singleton-support factorizations, then certainly $|A| \leq \tfrac{1}{2}e$.  Otherwise,
$p$ has at most $\lfloor\tfrac{1}{2}(e-1)\rfloor = \tfrac{1}{2}e - 1$ factorizations supported on 2 atoms,
so $|A| \leq \tfrac{1}{2}e$.

We next prove part~(a).  Since $p = 2q$ has at most two solutions for $q\in \ZZ_m$, $p$ has at most two singleton-support factorizations.  This means $p$ has either zero, one, or two singleton-support factorizations, and proceeding as before in each case, one readily shows $|A| \le \lfloor \tfrac{1}{2}e\rfloor + 1$.  

This leaves part~(c).  If $p = 2q$ has no solutions $q \in \ZZ_m$, then all factorizations in $A$ are supported on 2 atoms, and
there can be at most $|A| \le \lfloor \tfrac{1}{2}e \rfloor$ such factorizations.
\end{proof}

\begin{lemma}\label{l:nocovers}
Fix a non-degenerate face $F \subseteq \Cm$, and let $N$ be its Kunz nilsemigroup and $e = \mathsf e(N) = m-3$.  If some non-nil element of $N$ has a factorization of coordinate sum $3$, then $\dim(F)\geq \lfloor \tfrac{1}{2}(m-3)\rfloor.$
\end{lemma}

\begin{proof}
Suppose $p,q \in N \setminus \{0, \infty\}$ are the only two non-atoms.  If $q$ has a factorization with coordinate sum 3, then $q = p + a$ for some atom $a$.  Moreover, if $p = a_1 + a_2$ is any factorization of $p$, then $q = a_1 + a_2 + a$, and since $a_1 + a \ne q$ and $a_2 + a \ne q$ are not atoms, $p = a_1 + a = a_2 + a$.  Partial cancellativity implies $a_1 = a_2 = a$.  In~particular, $q = 3a$ is the only factorization of $q$ with coordinate sum 3, and $p = 2a$ is the only factorization of $p$.  

If $m$ is even, $q$ has at most $\tfrac{1}{2}(m-4) + 1$ coordinate sum $2$ factorizations by Lemma~\ref{l:coordsum}(a), so $|\mathsf{Z}_N(q)| \le \tfrac{1}{2}(m-4) + 2 = \tfrac{1}{2}m$. If $m$ is odd, $q$ has at most $\tfrac{1}{2}(m-3)$ coordinate sum 2 factorizations by Lemma~\ref{l:coordsum}(b), so $|\mathsf{Z}_N(q)| \le \tfrac{1}{2}(m-1)$.  In either case, $|\mathsf{Z}_N(q)| \leq \lfloor \tfrac{1}{2}m \rfloor$, and 
\[
\dim(F)
\ge (m - 3) - (\lfloor\tfrac{1}{2}m\rfloor - 1)
= \lfloor\tfrac{1}{2}(m-3)\rfloor
\]
then follows from Lemma~\ref{l:lemmadimbound}.  
\end{proof} 

\begin{thm}\label{t:edimbounds}
Fix a multiplicity $m \geq 7$ and a non-degenerate face $F \subseteq \Cm$, and let $e = \mathsf e(F)$ and $d = \dim F$.  Let $N$ denote the Kunz nilsemigroup of $F$.  
\begin{enumerate}[(a)]
\item 
If $e = m - 1$, then $d = m - 1$.  

\item 
If $e = m - 2$, then $\dfloor \le d \le e$.  

\item 
If $e = m - 3$ and $m$ is odd, then $2 \le d \le e$.

\end{enumerate}
\end{thm}

\begin{proof}
By Theorem~\ref{t:dimension}, $d \le e$.  If $e = m-1$, each non-nil element of $N$ is an atom and therefore has a single factorization. Applying Lemma~\ref{l:lemmadimbound}, we have $d = m-1$.

If $e = m-2$, there is a unique non-atom $p \in N \setminus \{0, \infty\}$.  Every factorization of $p$ has coordinate sum 2, so $|\mathsf{Z}_N(p)|\leq \lfloor \tfrac{1}{2}(m-2)\rfloor + 1$ by Lemma~\ref{l:coordsum}.  Lemma~\ref{l:lemmadimbound} yields
\[
d \geq (m-2) - \lfloor \tfrac{1}{2}(m-2) \rfloor = \dfloor.
\]

Lastly, suppose $e = m - 3$ and $m$ is odd, and let $p,q \in N \setminus \{0, \infty\}$ be the non-atoms.  
If one of $p$ or $q$ has a factorization of coordinate sum $3$, Lemma~\ref{l:nocovers} gives 
\[
d \ge \lfloor \tfrac{1}{2}(m-3) \rfloor \ge 2
\]
since $m \geq 7$.
Otherwise, all factorizations of $p$ and $q$ have coordinate sum 2, Lemma~\ref{l:coordsum} gives $|\mathsf{Z}_N(p)| \le \tfrac{1}{2}(m-3)$ and $|\mathsf{Z}_N(q)| \leq \tfrac{1}{2}(m-3)$.  This yields
\[
d \geq (m-3) - (m-5) = 2,
\]
after applying Lemma~\ref{l:lemmadimbound}.  
\end{proof}

\begin{prop}\label{p:nonapery}
Fix $m \ge 7$.  In each of the following cases, if $F \subseteq \Cm$ is a non-degenerate face with $\dim F = d$ and $\mathsf e(F) = e$, then $m$ is even and $F$ is not Ap\'ery:
\begin{enumerate}[(a)]
\item 
$d = \tfrac{1}{2}m - 1$ and $e = m - 2$; or

\item 
$d = 1$ and $e = m - 3$.

\end{enumerate}
\end{prop}

\begin{proof}
Let $N$ be the Kunz nilsemigroup of $F$.  First, suppose $d = \tfrac{1}{2}m - 1$ and $e = m - 2$, and let $p \in N \setminus \{0, \infty\}$ be the unique non-atom.  By Lemmas~\ref{l:lemmadimbound} and~\ref{l:coordsum}(a), 
\[
\tfrac{1}{2}m
= e - d + 1
\le |\mathsf{Z}_N(p)|
\le \lfloor \tfrac{1}{2}e \rfloor + 1
= \lfloor \tfrac{1}{2} m \rfloor
\]
so $|\mathsf{Z}_N(p)| = \frac{1}{2}m = \tfrac{1}{2}e + 1$.  Necessarily, $p$ has two singleton-support factorizations, i.e., $p = 2a = 2a'$ for distinct atoms $a, a' \in N$, and so by Corollary~\ref{cor-congruence-problem}, $F$ is not Ap\'ery.

Next, suppose $d = 1$ and $e = m - 3$, and let $p,q \in N \setminus \{0, \infty\}$ denote the non-atoms.  By Lemma~\ref{l:nocovers}, every factorization of $p$ and $q$ has coordinate sum $2$.
As in part~(a), 
\[
m - 2
= e - d + 2
\le |\mathsf{Z}_N(p)| + |\mathsf{Z}_N(q)|
\le 2(\lfloor \tfrac{1}{2}e \rfloor + 1)
= 2\lfloor \tfrac{1}{2} (m - 1) \rfloor.
\]
By Lemma~\ref{l:coordsum}(b) and the first inequality, $m$ is even.  This forces 
\[
|\mathsf{Z}_N(p)| + |\mathsf{Z}_N(q)| = m - 2
\qquad \text{and thus} \qquad
|\mathsf{Z}_N(p)| = |\mathsf{Z}_N(q)| = \tfrac{1}{2}m - 1 = \tfrac{1}{2}(e + 1).
\]
There are two cases.  First, if $p \ne 2q \in \ZZ_m$, then the atom $p - q$ does not appear in any factorization of $p$, meaning $p$ has two singleton-support factorizations and $F$ is not Ap\'ery by Corollary~\ref{cor-congruence-problem}.  As such, assume $p = 2q$ and $q = 2p$.  This forces $p = 2d$ and $q = 4d$ for some $d \in \ZZ_m$ with order 6.  For~each atom $i \in N$, let the standard basis vector $e_i$ denote the factorization of $i$.  Letting $I = [1,3d-1]\setminus \{d\}$,
\[
\rho=\{(e_{d+i}+e_{d-i},2e_d):i\in I\}\cup\{(e_{5d+j}+e_{5d-j},2e_{5d}):j\in I\}
\]
is a presentation of $N$. Letting $\tau_i = (e_{d+i} + e_{d-i}) - 2e_{d}$ and $\omega_j = (e_{5d+j} + e_{5d-j}) - 2e_{5d}$, 
\[
\sum_{i\in I}\tau_{i}-\sum_{j\in I}\omega_{j}=(m-3)e_{5d}-(m-3)e_d = (m - 3)(e_{5d} - e_d) \in L_N
\] 
since $e_i$ appears exactly once for each $i \ne d, 5d$.  By Corollary \ref{cor-congruence-problem}, $F$ is not Ap\'ery.
\end{proof}

\begin{remark}\label{r:nonaperyray}
Theorem~\ref{t:attainable}(b) exhibits faces of the Kunz cone described in Proposition~\ref{p:nonapery}(a).  The Kunz nilsemigroup with divisibility poset depicted in Figure~\ref{f:nonaperyray} corresponds to a face described by Proposition~\ref{p:nonapery}(b); it is not difficult to construct similar nilsemigroups for any even $m \ge 8$.
\end{remark}



\begin{figure}[t!]
\begin{center}
\begin{subfigure}[b]{0.18\textwidth}
\begin{center}
\includegraphics[scale=0.25]{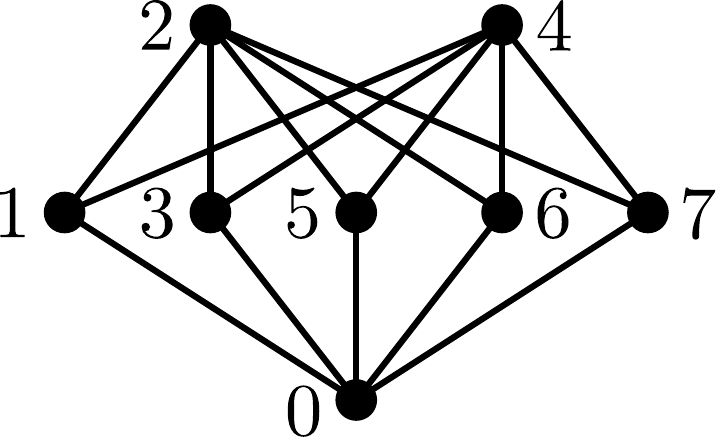}
\end{center}
\caption{}
\label{f:nonaperyray}
\end{subfigure}
\hspace{0.04\textwidth}
\begin{subfigure}[b]{0.19\textwidth}
\begin{center}
\includegraphics[scale=0.25]{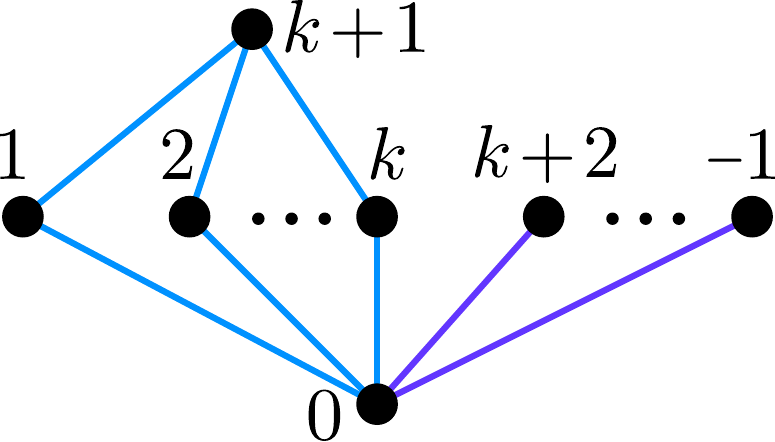}
\end{center}
\caption{}
\label{f:attainablea}
\end{subfigure}
\hspace{0.04\textwidth}
\begin{subfigure}[b]{0.18\textwidth}
\begin{center}
\includegraphics[scale=0.25]{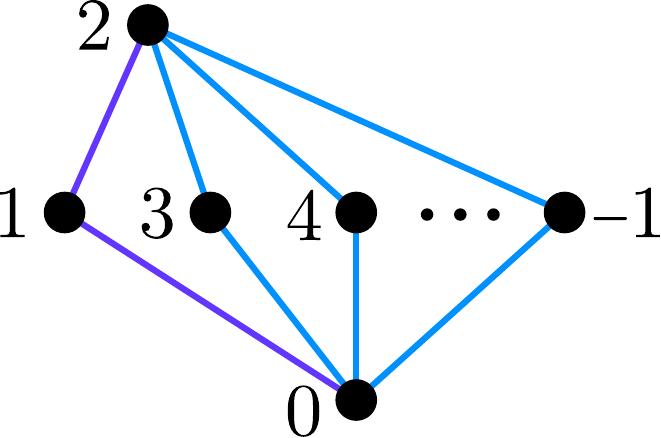}
\end{center}
\caption{}
\label{f:attainableb}
\end{subfigure}
\hspace{0.04\textwidth}
\begin{subfigure}[b]{0.27\textwidth}
\begin{center}
\includegraphics[scale=0.25]{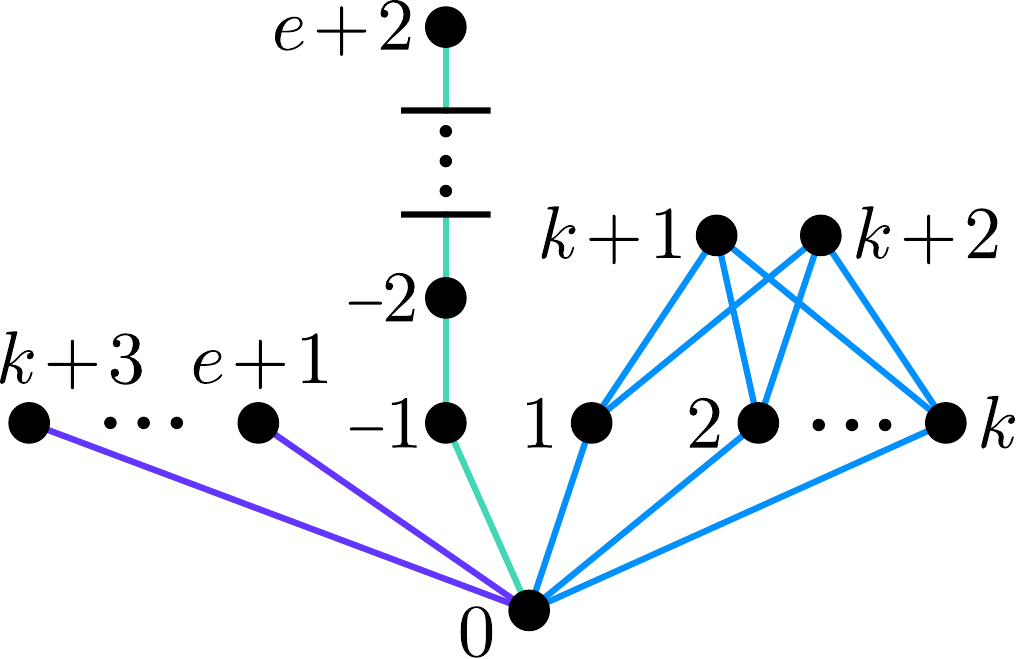}
\end{center}
\caption{}
\label{f:attainablec}
\end{subfigure}
\end{center}
\caption{Divisibility poset of a Kunz nilsemigroup corresponding to a non-Ap\'ery face with dimension $1$, along with depictions of the three constructions for Theorem \ref{t:attainable}.}
\label{f:attainable}
\end{figure}


\begin{thm}\label{t:attainable}
Fix $m \geq 7$.  There is a face $F \subseteq \Cm$ with $\mathsf e(F) = e$ and $\dim F = d$ if:
\begin{enumerate}[(a)]
\item 
$e = m - 1$ and $d = m-1$;

\item 
$e = m-2$ and $\dfloor \le d \le m-2$; or

\item
$2 \leq e \leq m-3$ and $d \in [2,e]$.

\end{enumerate}
Each such face $F$ can be chosen Ap\'ery, except where impossible by Proposition~\ref{p:nonapery}.  
\end{thm}

\begin{proof}
The face $F = \Cm$, which has the numerical semigroup $S = \<m, m+1,\ldots, 2m-1\>$, verifies part~(a).  
Next, for each $k = 1, \ldots, m-2$, consider the point $x \in \ZZ_{\ge 0}^{m-1}$ with 
\[ 
x_i =
\begin{cases} 
2 & \text{if } i = 1, \ldots, k; \\
4 & \text{if } i = k+1; \\
3 & \text{if } i = k+2, \ldots, m-1.
\end{cases}
\]
Since $x_i + x_j \ge x_k$ for any $i, j, k \in \ZZ_m \setminus \{0\}$, we have $x \in \Cm$.  Moreover, if $x_i + x_j = x_{i+j}$, then $x_i = x_j = 2$ and $x_{i+j} = 4$, which necessitates $i \in [1,k]$ and $j = k + 1 - i$.  As~such, the Kunz nilsemigroup $N$ of the face $F \subseteq \Cm$ with $x \in F^\circ$ has divisibility poset depicted in Figure~\ref{f:attainablea}, so $\dim F = e - (|\mathsf Z_N(k+1)| - 1) = m - 2 - \lfloor \tfrac{1}{2}(k-1) \rfloor$ by Theorem~\ref{t:dimension}.  Additionally, $k + 1$~has at most one singleton-support factorization, so letting $\hat z \in \mathsf Z_N(k+1)$ denote a factorization of $k+1$ with minimal support, 
\[
\rho = \{(\hat{z},z):z \in \mathsf{Z}_N(k+1), \, z \ne \hat z\}
\]
is a presentation of $N$.  Since each factorization $z$ above has support 2, after permuting rows and columns $M_\rho = [I \mid A]$ for some matrix $A$, meaning $L_N$ is saturated.  
%
%
%

As we vary $k$ in the above construction, we obtain faces whose dimension $d$ has $\lfloor \tfrac{m}{2}\rfloor \le d \le m-2$.  If $m$ is odd, $\lfloor \tfrac{m}{2}\rfloor = \tfrac{1}{2}(m-1)$, so to prove part~(b), it remains to assume $m$ is even and construct a face $F$ with $\dim F = \tfrac{1}{2}m - 1$.  To this end, 
consider the point $x \in \ZZ_{\ge 0}^{m-1}$ with
\[ 
x_i =
\begin{cases} 
1 & \text{if } i = 1; \\
2 & \text{if } i = 2; \\
1 & \text{if } i = 3, \ldots, m-1.
\end{cases}
\]
As before, it is clear $x \in \Cm$, and $x_i + x_j = x_{i+j}$ only preicsely when $i + j = 2$, so the Kunz nilsemigroup of the face $F \subseteq \Cm$ with $x \in F^\circ$ has divisibility poset depicted in Figure~\ref{f:attainableb}, and thus $\dim F = e - (|\mathsf Z_N(2)| - 1) = \tfrac{1}{2}m - 1$ by Theorem~\ref{t:dimension}.  As $F$ satisfies Proposition~\ref{p:nonapery}(a), part~(b) is proved.  
%
%

This leaves part~(c).  First, if $d = 2$, then by \cite[Theorem~4.2]{kunzfaces2}, the numerical semigroup $S = \<m,m+1, \ldots, m + e\>$ lies in a 2-dimensional face of $\Cm$ for $e \in [2, m - 3]$, so it suffices to consider $e \ge d \ge 3$.  To this end, fix $e \in [3, m-3]$ and fix $k \in [2, e - 1]$, 
let $h = m - e - 1$, and consider the point $x \in \ZZ_{\ge 0}^{m-1}$ with 
\[ 
x_i =
\begin{cases} 
h & \text{if } i = 1, \ldots, k; \\
2h & \text{if } i = k+1,k+2; \\
2h-1 & \text{if } i = k+3, \ldots, e+1; \\
2(m - i) & \text{if } i = e+2, \ldots, m-1.
\end{cases}
\]
Certainly $x_i + x_j \ge x_{i+j}$ whenever $x_i \ge h$ and $x_j \ge h$, and it is not hard to check that $x_{-1} + x_j = 2 + x_j \ge x_{j-1}$ for each $j$.  As such, if $x_i < h$, then $i > e + 2$, so 
\[
x_i + x_j = (m - i)x_{-1} + x_j \ge (m - i - 1)x_{-1} + x_{j-1} \ge \cdots \ge x_{i+j}.
\]
This proves $x \in \Cm$.  Moreover, $x_i + x_j = x_{i+j}$ only occurs when (i) $i \in [1,k]$ and either $j = k+1 - i$ or $k + 2 - i$, or (ii) $i, j, i + j \in \{e + 2, \ldots, m-1\}$.  In particular, the Kunz nilsemigroup of the face $F \subseteq \Cm$ with $x \in F^\circ$ has divisibility poset depicted in Figure~\ref{f:attainablec}.  As such, $\mathsf e(F) = k + ((e + 1) - (k + 2)) + 1 = e$.  To obtain $\dim F$, 
consider the face $F' \subseteq \mathcal C_{k+3}$ containing the arithmetical numerical semigroup
\[
S = \<k+3, k+4, k+5, \ldots, 2k+3\>.
\]
The divisibility poset of the Kunz nilsemigroup $N'$ of $F'$ is identical to the divisibility poset depicted in Figure~\ref{f:attainablec} restricted to the elements in $\{0, 1, \ldots, k+2\}$.  As such, there exist presentations $\rho$ and $\rho'$ of $N$ and $N'$, respectively, such that 
\[
M_\rho=[M_{\rho'} \mid 0 \, ]
\qquad \text{and thus} \qquad
L_N = L_{N'} \times \{0\}^{e-k}.
\]
In particular, $L_N$ is saturated since $L_{N'}$ is, and $\rk(M_\rho) = \rk(M_{\rho'})$, so 
\[
\dim F
= e - \rk(M_\rho)
= e - \rk(M_{\rho'})
= e - (k - \dim F')
= e - k + 2,
\]
by Theorem~\ref{t:dimension} and \cite[Theorem~4.2]{kunzfaces2}.  
\end{proof}

\end{document}